\documentclass{amsart}%
\usepackage{amsfonts}
\usepackage{amssymb}
\usepackage{amsmath}
\usepackage{graphicx}%
\newtheorem{theorem}{Theorem}[section]
\theoremstyle{plain}

\newtheorem{corollary}[theorem]{Corollary}

\newtheorem{remark}{Remark}[section]

\numberwithin{equation}{section}
\begin{document}
\title[FACTORIZATION OF ELLIPTIC\ DIVISIBILITY SEQUENCES\ ]{FACTORIZATION OF ELLIPTIC\ DIVISIBILITY SEQUENCES\ VIA\ MODULAR\ CURVES}
\author{BET\"{U}L GEZER, \.{I}PEK\ \c{C}OLAK, OSMAN B\.{I}Z\.{I}M}
\address{Uludag University, Faculty of Science, Department of Mathematics,
G\"{o}r\"{u}kle, 16059, Bursa-TURKEY}
\email{betulgezer@uludag.edu.tr, ipekk3353@gmail.com, obizim@uludag.edu.tr}
\date{06. 07. 2026.}
\maketitle

\begin{abstract}
In this paper, we give a structural factorization of elliptic divisibility
sequences in terms of finitely many modular functions on the modular curve
$X_{1}(N)$. More precisely, we express the functions obtained by pulling back
division polynomials along the universal torsion section as products of fixed
modular factors. Then we show that this factorization is compatible with the
natural action of $(\mathbb{Z}/N\mathbb{Z})^{\times}$ on the modular curve
$X_{1}(N)$ and so descends to the modular curve $X_{0}(N)$ via the norm map.
As applications, we obtain a perfect power criterion, local and global growth
estimates for the associated elliptic divisibility sequences.

\end{abstract}
\tableofcontents

\section{Introduction}

Let $E$ be an elliptic curve defined over a field $K$ given by a Weierstrass
equation%
\begin{equation}
E:y^{2}+a_{1}xy+a_{3}y=x^{3}+a_{2}x^{2}+a_{4}x+a_{6}\text{, } \label{111}%
\end{equation}
let $E(K)$ denote the group of $K$-rational points on $E$, and let
$\mathcal{O}$ denote the point at infinity, the identity element of the group
$E(K)$. For background on elliptic curves, see \cite{JS7,JS1, JS2}. The
$n$-\textit{division polynomial }$\psi_{n}$\textit{\ }of $E$\textit{\ }is the
unique function $\psi_{n}$ $\in K(E)$ with divisor%
\[
\operatorname{div}(\psi_{n})=[n]^{\ast}(\mathcal{O})-n^{2}(\mathcal{O}).
\]
Division polynomials play a central role in the theory of elliptic curves,
e.g. they appear in height theory \cite{JS1}, in reduction theory \cite{JS7},
in algorithms for counting points on elliptic curves over finite fields
\cite{ReneS}, and in the theory of elliptic divisibility sequences (EDSs)
\cite{MW1}. See also \cite{JS3}, \cite{CH}, \cite{SP}.

The sequence of values $(\psi_{n}(P))_{n\geq0}$ of the division polynomials
evaluated at a point $P$ satisfies the classical elliptic divisibility
recurrence relations, see \cite{MW1}. Hence, this gives rise to an elliptic
divisibility sequence $(W_{n})_{n\geq0~}$associated to the pair $(E,P)$ where
$W_{n}=\psi_{n}(P)$. The sequence $(W_{n})_{n\geq0}$ is called the
\textit{elliptic divisibility sequence }associated to the pair $(E,P)$.

Pairs $(E,P)$ with prescribed torsion structure are parametrized by modular
curves, providing a geometric way to study families of elliptic curves
together with torsion points. This approach leads to a modular interpretation
of elliptic divisibility sequences. Furthermore, we use the modular curves
$Y_{1}(N)$ and $X_{1}(N)$ to analyze the sequences $(W_{n})_{n\geq0}$, where
affine curve $Y_{1}(N)$ parametrizes pairs $(E,P)$ with a point $P$ of order
$N$, and $X_{1}(N)$ denotes its smooth projective compactification.

The aim of this paper is to interpret elliptic divisibility sequences
$(W_{n})_{n\geq0}$ associated to elliptic curves together with torsion points
in terms of modular functions on $Y_{1}(N)$. The terms of an EDS can therefore
be viewed as the specialization at $(E,P)$ of modular functions on $Y_{1}(N)$.
Thus they are not merely quantities depending on a chosen Weierstrass
equation, they arise from universal modular functions on $Y_{1}(N)$. We make
this interpretation explicit and show that these functions admit finite
modular factorizations with quadratic exponents depending only on $n$ modulo
$N$.

Furthermore, if we want these sequences to depend only on the elliptic curve
over $K$ and the point $P$ $\in E(K)$, but independent of the chosen model, we
are naturally led to a canonical definition of elliptic divisibility
sequences, obtained by pulling back the division polynomials $\psi_{n}$ along
a section. More precisely, there exists a universal elliptic curve
$\mathcal{E\rightarrow}$ $Y_{1}(N)$ with universal section $\mathcal{P}$ $:$
$Y_{1}(N)~\mathcal{\rightarrow}$\ $\mathcal{E}$. Pulling back the $n$-th
division polynomial $\psi_{n}$ along the section $\mathcal{P}$ defines a
modular function
\begin{equation}
f_{n}=\mathcal{P}^{\ast}(\psi_{n})\in K(Y_{1}(N)) \label{1}%
\end{equation}
where $K(Y_{1}(N))$ is the function field of the modular curve $Y_{1}(N)$, if
$n\mathcal{P}=\mathcal{O}$, then $f_{n}$ vanishes identically. Now let
$Y_{1}(N)(K)$ denote the $K$-rational points on $Y_{1}(N)$, then for any
$(E,P)\in Y_{1}(N)(K)$, evaluating $f_{n}$ at $(E,P)$ recovers the associated
elliptic divisibility sequence, i.e.,
\[
W_{n}=f_{n}(E,P).
\]
Elliptic divisibility sequences over function fields have been studied by
Ingram et al. \cite{IP} and others \cite{CR,N1,N2}. In the present work, we
consider the functions $f_{n}$ arising from the universal elliptic curve over
$Y_{1}(N)$ together with its universal section. This allows elliptic
divisibility sequences to be interpreted in terms of modular functions and the
geometry of modular curves.

In \cite{BG}, elliptic divisibility sequences associated to elliptic curves
over $%
%TCIMACRO{\U{211a} }%
%BeginExpansion
\mathbb{Q}
%EndExpansion
$ were studied for specific torsion structures using explicit models of
elliptic curves. By viewing these sequences as specializations of the
functions $f_{n}$ at points of $Y_{1}(N)$, we extend these results to elliptic
curves over arbitrary number fields and to arbitrary levels $N$. We further
carry out these results on the modular curves $X_{0}(N)$, which parametrize
elliptic curves together with cyclic subgroups $C$ of order $N$. Thus, we
obtain a geometric interpretation of elliptic divisibility sequences both in
terms of points and cyclic isogenies.

Our first main result gives a structural factorization of EDSs in terms of
finitely many modular functions on $X_{1}(N)$.

\begin{theorem}
Let $N\geq2$, let $Y_{1}(N)$ be the affine modular curve parametrizing
elliptic curves together with a point of order $N$, and let $X_{1}(N)$ be its
smooth compactification with the function field $K(X_{1}(N))$ over a number
field $K$. Suppose that $(E,P)$ correspond to a $K$-rational point of
$Y_{1}(N)$, and let $(W_{n})_{n\geq0}$ be the elliptic divisibility sequence
associated to $(E,P)$ via division polynomials $\psi_{n}$. Then there exist
finitely many modular functions $g_{1},\dots,g_{r}\in K(X_{1}(N))^{\times}$
depending only on $N$ and on the chosen model of $X_{1}(N)$, such that
\begin{equation}
W_{n}=\prod_{i=1}^{r}g_{i}^{\alpha_{i}(n)}(E,P)\text{ \ \ \ \ \ \ } \label{2}%
\end{equation}
for all $n\geq1$ such that $N\nmid n$, here each $\alpha_{i}(n)$ is a
polynomial in $n$ of degree at most $2$ whose values depend only on the
residue class of $n$ modulo $N$, and each function $g_{i}$ has a divisor
supported on the cusps and a finite set of points determined by the chosen
model of $X_{1}(N)$.
\end{theorem}

The proof relies on the existence of a universal elliptic curve $\mathcal{E}$
$\mathcal{\rightarrow}$ $Y_{1}(N)$ with a universal torsion section
$\mathcal{P}$ of exact order $N$. Pulling back the division polynomials
$\psi_{n}$ along the section $\mathcal{P}$\ we obtain functions on $Y_{1}(N)$
and extend these functions to rational functions on $X_{1}(N)$. Then we use
divisor theory and evaluate these functions at $(E,P)$ to express $W_{n}$ in
terms of these modular functions.

We can extend Theorem 1.1 to the modular curves $X_{0}(N)$, which parametrize
elliptic curves together with cyclic subgroups $C$ of order $N$. Since
$X_{0}(N)$ parametrizes cyclic subgroups without choice of generator, we
choose a generator $P$ $\in$ $C$ to lift $(E,C)$ to a point of $X_{1}(N)$; see
\cite[Section 1.5]{DS}. This lifting is well-defined up to the natural action
of $(%
%TCIMACRO{\U{2124} }%
%BeginExpansion
\mathbb{Z}
%EndExpansion
/N%
%TCIMACRO{\U{2124} }%
%BeginExpansion
\mathbb{Z}
%EndExpansion
)^{\times}$. Finally, by taking suitable products over all generators of $C$,
we obtain canonical functions on $X_{0}(N)$. Then we show that the
factorization structure descends to this setting.

The following theorem shows that the evaluation of the modular function
$f_{n}$ in (\ref{1}) at a lifted point gives an elliptic divisibility sequence
and that the factorization structure extends to this case.

\begin{theorem}
With notation and assumptions as in Theorem 1.1, let $f_{n}$ be as in
(\ref{1}). Let $(E,C)$ be a point of $Y_{0}(N)(K)$, where $C$ is a cyclic
subgroup of order $N$, and choose a generator $P\in C$, yielding a lift
$(E,P)\in Y_{1}(N)(K)$. Then the sequence%
\[
W_{n}=f_{n}(E,P)
\]
is an elliptic divisibility sequence. Moreover, this sequence satisfies the
factorization in (\ref{2}).
\end{theorem}

Note that choosing a generator $P\in C$ lifts $(E,C)$ to a point of $Y_{1}(N)
$. This is compatible with the natural action of $(\mathbb{Z}/N\mathbb{Z}%
)^{\times}$ and so descends to $X_{0}(N)$ via the norm map. The next theorem
describes this descent in terms of modular functions.

\begin{theorem}
With notation and assumptions as in Theorems 1.1 and 1.2, define%
\[
F_{n}=\mathrm{Nm}_{K(X_{1}(N))/K(X_{0}(N))}(f_{n})\text{,}%
\]
where $\mathrm{Nm}_{K(X_{1}(N))/K(X_{0}(N))}$ denotes norm for $K(X_{1}%
(N))/K(X_{0}(N))$. Then the function $F_{n}$ in $K(X_{0}(N))^{\times} $ and
for any $(E,C)\in Y_{0}(N)(K)$ and any generator $P\in C$, we have%
\[
F_{n}(E,C)=\prod_{k\in(\mathbb{Z}/N\mathbb{Z})^{\times}}f_{n}(E,kP).
\]
Furthermore, the functions $F_{n}$ satisfy the following factorization
\[
F_{n}=\prod_{i=1}^{r}G_{i}^{\alpha_{i}(n)},
\]
where $G_{i}=\mathrm{Nm}_{K(X_{1}(N))/K(X_{0}(N))}(g_{i})\in K(X_{0}%
(N))^{\times}$.
\end{theorem}

\section{Division Polynomials and Elliptic Divisibility Sequences}

Recall that for an elliptic curve $E$ over a number field $K$ given by a
Weierstrass equation (\ref{111}) and a point $P=(x,y)\in E(K)$, the division
polynomials are defined using the initial values%
\begin{align}
\psi_{0}(P)  &  =0\text{,}\label{8}\\
\psi_{1}(P)  &  =1\text{,}\nonumber\\
\psi_{2}(P)  &  =2y+a_{1}x+a_{3}\text{,}\nonumber\\
\psi_{3}(P)  &  =3x^{4}+b_{2}x^{3}+3b_{4}x^{2}+3b_{6}x+b_{8}\text{,}%
\nonumber\\
\psi_{4}(P)  &  =\psi_{2}(P)(2x^{6}+b_{2}x^{5}+5b_{4}x^{4}+10b_{6}%
x^{3}+10b_{8}x^{2}\nonumber\\
&  \ \ +(b_{2}b_{8}-b_{4}b_{6})x+(b_{4}b_{8}-b_{6}^{2}))\text{,}\nonumber
\end{align}
together with the recursive relations
\begin{align}
\psi_{2n+1}(P)  &  =\psi_{n+2}(P)\psi_{n}(P)^{3}-\psi_{n-1}(P)\psi
_{n+1}(P)^{3}\text{, \ \ for }n\geq2\text{,}\nonumber\\
\psi_{2n}(P)  &  =(\psi_{n-1}(P)^{2}\psi(P)\psi_{n+2}(P)\\
&  \ \ -\psi_{n-2}(P)\psi_{n}(P)\psi_{n+1}(P)^{3})(\psi_{2}(P))^{-1}\text{,
\ \ for }n\geq3\text{.}\nonumber
\end{align}
where $b_{i}$'s are usual quantities in \cite[Chapter III]{JS1}. In particular
the sequence $(\psi_{n}(P))$ satisfies the recurrence relation (\ref{6}).

An \textit{elliptic divisibility sequence} (EDS) is a sequence $(W_{n}%
)_{n\geq0}$ satisfying a nonlinear recurrence%
\begin{equation}
W_{m+n}W_{m-n}=W_{m+1}W_{m-1}W_{n~}^{2}-W_{n+1}W_{n-1}W_{m~}^{2} \label{6}%
\end{equation}
and the divisibility property%
\[
W_{n}|W_{m}\text{ whenever }n|m
\]
for all $m\geq n\geq1$. See; \cite{MW1}, \cite{MW2}, \cite{GE1}, \cite{RS},
and \cite{CS} for background on EDSs. An elliptic divisibility sequence is
called \textit{proper} if $W_{0}=0$, $W_{1}=1$, and $W_{2}W_{3}\neq0$. A
proper elliptic divisibility sequence is called \textit{nonsingular} if its
discriminant is non zero; see \cite{JS3}, \cite{JS4} or \cite{MW1} for the
explicit formula of the discriminant of an EDS.

The arithmetic properties of these sequences were first studied by Ward
\cite{MW1, MW2}. Ward \cite[Theorem 12.1]{MW1} showed that nonsingular
elliptic divisibility sequences arise as values of the division polynomials of
an elliptic curve. More precisely, he proved that if $(W_{n})_{n\geq0}$ is a
nonsingular elliptic divisibility sequence, then there exist a lattice
$L\subset%
%TCIMACRO{\U{2102} }%
%BeginExpansion
\mathbb{C}
%EndExpansion
$ and a complex number $z$ $\in%
%TCIMACRO{\U{2102} }%
%BeginExpansion
\mathbb{C}
%EndExpansion
$ such that%
\[
W_{n}=\psi_{n}(z\text{, }L)=\frac{\sigma(nz\text{, }L)}{\sigma(z\text{,
}L)^{n^{2}}}\text{~\ \ for all }n\geq1\text{,}%
\]
where $\psi_{n}(z$, $L)$ and $\sigma(z$, $L)$ are the $n$-division polynomial
and the Weierstrass $\sigma$-function associated to the lattice $L$,
respectively. Subsequently, Silverman \cite[Proposition 18]{JS3} reformulated
Ward's result in an arithmetic setting over $%
%TCIMACRO{\U{211a} }%
%BeginExpansion
\mathbb{Q}
%EndExpansion
$, relating nonsingular elliptic divisibility sequences to division
polynomials evaluated at rational points on elliptic curves. In the present
paper, we work over an arbitrary number field $K$. Thus for an elliptic curve
$E$ over $K$ and a point $P\in E(K)$, we use the normalization
\begin{equation}
W_{n}=\psi_{n}(P)\text{ \ \ \ for all }n\geq1 \label{7}%
\end{equation}
where $\psi_{n}$ is the $n$-division polynomial on $E$. Furthermore, the
sequence $(W_{n})_{n\geq0}$ is the elliptic divisibility sequence associated
to the pair $(E,P)$.

We now relate these arithmetic properties to the modular framework used
throughout the paper. Let $\mathcal{E\rightarrow}$ $Y_{1}(N)$ be the universal
elliptic curve with universal section $\mathcal{P}$ $:$ $Y_{1}%
(N)~\mathcal{\rightarrow}$\ $\mathcal{E}$. Pulling back the $n$-th division
polynomial $\psi_{n}$ along $\mathcal{P}$ defines a modular function
\begin{equation}
f_{n}=\mathcal{P}^{\ast}(\psi_{n})\in K(Y_{1}(N))^{\times}%
\end{equation}
where $K(Y_{1}(N))$ is the function field of the modular curve $Y_{1}(N)$, if
$n\mathcal{P}=\mathcal{O}$, then $f_{n}$ vanishes identically. For a
$K$-rational point $(E,P)\in Y_{1}(N)(K)$, evaluating $f_{n}$ at $(E,P)$ gives
$\ $%
\[
f_{n}(E,P)=\psi_{n}(P)=W_{n}.
\]
Thus by evaluating $f_{n}$ at $(E,P)$, we recover the associated elliptic
divisibility sequence.

\section{Factorization and Applications}

In this section, we show the factorization of the functions $f_{n}%
=\mathcal{P}^{\ast}(\psi_{n})$ and obtain a perfect power criterion. Then we
discuss the growth of the sequence $(W_{n})_{n\geq0}=(f_{n}(E,P))_{n\geq0}$.
We illustrate the results with an explicit example in the case $N=13$. Further
examples for other levels, derived from the models presented in \cite[pp.
1144-1145]{SU}, which provide relatively simple factorization formulas for
elliptic divisibility sequences, are given in the Appendix.

\subsection{Proof of Theorem 1.1}

\begin{proof}
[Proof of Theorem 1.1]Let $\mathcal{E}\rightarrow Y_{1}(N)$ denote the
universal elliptic curve with universal torsion section $\mathcal{P}$ of exact
order $N$. For each integer $n\geq1$, we consider the $n$-th division
polynomial $\psi_{n}$ as a rational function on $\mathcal{E}$ with divisor
\[
\operatorname{div}(\psi_{n})=\sum_{\mathcal{T}\in\mathcal{E}[n]}%
(\mathcal{T})-n^{2}(\mathcal{O})\text{.}%
\]
Pulling back $\psi_{n}$ via the universal section $\mathcal{P}$, we obtain a
rational function%
\[
f_{n}=\mathcal{P}^{\ast}(\psi_{n})\in K(Y_{1}(N))^{\times}%
\]
whose divisor is given by
\[
\operatorname{div}(f_{n})=\operatorname{div}(\mathcal{P}^{\ast}(\psi
_{n}))=\mathcal{P}^{\ast}(\operatorname{div}(\psi_{n}))=\sum_{\mathcal{T}%
\in\mathcal{E}[n]}\mathcal{P}^{\ast}(\mathcal{T})-n^{2}\mathcal{P}^{\ast
}(\mathcal{O})\text{.}%
\]
Since $\mathcal{P}$ has exact order $N$, an equality $\mathcal{P}$ $=$
$\mathcal{T}$ for some $\mathcal{T}\in\mathcal{E}[n]$ is possible if and only
if $N~|~n$. If $N~|~n$, then $f_{n}$ vanishes identically and its divisor is
undefined. Thus we assume that $N\nmid n$. In this case, $\mathcal{P}$ never
intersects the $n$-torsion subgroup over $Y_{1}(N)$, thus $f_{n}$ has no zeros
or poles on this curve. Consequently, $f_{n}$ is a modular unit, and the
support of $\operatorname{div}(f_{n})$ is contained in the finite set of cusps
of $X_{1}(N)$. Moreover, the multiplicities in $\operatorname{div}(\psi_{n})$
depend quadratically on $n$. Pulling back along the section $\mathcal{P}$
preserves this growth. Therefore, the coefficients arising in the divisor
decomposition of $f_{n}$ at the cusps are polynomials in $n$ of degree at most
$2$. For certain residue classes modulo $N$, ramification at the cusps may
introduce cancellations that eliminate the quadratic term, reducing these
coefficients to linear or constant polynomials.

Now passing to the smooth compactification $X_{1}(N)$, the divisors appearing
above are supported on a finite set of $X_{1}(N)$ consisting of cusps and
finitely many non-cuspidal points arising from the chosen model. The group of
divisors supported on this finite set is finitely generated. Now we choose
divisors $D_{1},\dots,D_{r}$ generating the subgroup of principal divisors
$\operatorname{div}(f_{n})$. Then there exist polynomials $\alpha_{i}(n)$ in
$n$ of degree at most $2$ taking integer values on suitable residue classes
$n$ modulo $N$, such that
\[
\operatorname{div}(f_{n})=\sum_{i=1}^{r}\alpha_{i}(n)D_{i}.
\]

Since each $D_{i}$ is principal, we can choose functions $g_{i}\in
K(X_{1}(N))^{\times}$ with $\operatorname{div}(g_{i})=D_{i}$, where the
divisors of these $g_{i}$ are supported on the cusps along with a finite set
of non-cuspidal points arising from the coordinate singularities of the chosen
model. Now define%
\[
F_{n}=\prod_{i=1}^{r}g_{i}^{\alpha_{i}(n)}\text{.}%
\]
Thus $\operatorname{div}(F_{n})=\operatorname{div}(f_{n})$, which implies
$\operatorname{div}(F_{n}/f_{n})=0$. This means that $F_{n}/f_{n}$ $=c_{n}$
for some $c_{n}\in K^{\times}$, yielding $F_{n}=c_{n}f_{n}$. By identifying
$c_{n}f_{n}$ with $f_{n}$, we finally obtain $F_{n}=f_{n}$. Then evaluating
each $f_{n}$ at a $K $-rational point $(E,P)$ of $Y_{1}(N)$, we have
\[
W_{n}=f_{n}(E,P)=\prod_{i=1}^{r}g_{i}^{\alpha_{i}(n)}(E,P).
\]

\end{proof}

\begin{remark}
The bound $\deg(\alpha_{i}(n))\leq2$ reflects the quadratic growth of the
division polynomials on the universal elliptic curve. Geometrically, degree
drops may be interpreted in terms of ramification behavior of the section
$\mathcal{P}$ and special torsion relations among the points $nP$ (for an
explicit example, see the case $N=18$ in the Appendix, where $\alpha
_{2}(n)=1,$ if $n$ $\equiv9~(18)$ or $\alpha_{2}(n)=0,$ otherwise). Thus,
lower-degree exponents should be viewed as a natural feature of the geometry
rather than an exception.
\end{remark}

As an illustration of Theorem 1.1, we give an explicit factorization of the
terms of the sequence for $N=13$, see the Appendix for additional examples.
Let $t_{(u,v)}\in Y_{1}(13)(K)$ and let $(E_{t},P_{t})$ be the corresponding
pair, where $E_{t}$ is the associated elliptic curve and $P_{t}=(0,0)$. Then
\[
W_{n}(t)=W_{n}(u,v)=f_{n}(E_{t},P_{t})
\]
admits the following factorization.

\begin{theorem}
Let $N=13$, let $K$ be a number field, and let $Y_{1}(13)$ denote the affine
model
\begin{equation}
v^{2}+(u^{3}+u^{2}+1)v-u^{2}-u=0\text{,} \label{3}%
\end{equation}
and associate to each point $t_{(u,v)}\in Y_{1}(13)(K)$ the pair $(E_{t}%
,P_{t})$ where $E_{t}$ is the curve over $K(t)$ given by the equation
\begin{equation}
E_{t}:y^{2}+(1-c)xy-by=x^{3}-bx^{2} \label{10}%
\end{equation}
with
\[
c=s(r-1),b=cr\text{ and }r=1-uv,s=1-uv/(v+1)
\]
and where $P_{t}=(0,0)$ is a point of order 13 (see \cite[pp. 1144-1145]{SU},
for models and birational maps). Then%
\begin{equation}
W_{n}(t)=\left\{
\begin{array}
[c]{cc}%
{\small \varepsilon}\prod_{i=1}^{5}g_{i}^{\alpha_{i}(n)}(t) & \text{if
}13\nmid n\\
0 & \text{if }13|n\text{,}%
\end{array}
\right.  \text{ } \label{4}%
\end{equation}
where
\[
g_{1}(t)=u,g_{2}(t)=v,g_{3}(t)=uv-1,g_{4}(t)=uv-v-1,g_{5}(t)=v+1
\]
and
\[
\alpha_{1}(n)=(5n^{2}-p_{1})/13,\alpha_{2}(n)=(6n^{2}-p_{2})/13,\alpha
_{3}(n)=(4n^{2}-p_{3})/13,
\]%
\[
\alpha_{4}(n)=(5n^{2}-p_{4})/13,\alpha_{5}(n)=-(5n^{2}-p_{5})/13,
\]
with%
\[
{\small \varepsilon=}\left\{
\begin{array}
[c]{ll}%
{\small +1} & \text{if }{\small n}\text{ }{\small \equiv
1,2,3,10,11,12,17,18,19,20,21,22~(26)}\text{ \ \ }{\small \ }\\
{\small -1} & \text{if }{\small n}\text{ }{\small \equiv
4,5,6,7,8,9,14,15,16,23,24,25~(26)}\text{,}%
\end{array}
\right.
\]%
\[
{\small p}_{1}{\small =}\left\{
\begin{array}
[c]{cc}%
5 & \text{if }{\small n}\text{ }{\small \equiv1,12~(13)}\text{ \ }\\
7 & \text{if }{\small n}\text{ }{\small \equiv2,11~(13)}\text{ \ }\\
6 & \text{if }{\small n}\text{ }{\small \equiv3,10~(13)}\text{ \ }\\
2 & \text{if }{\small n}\text{ }{\small \equiv4,9~(13)}\text{ \ }\\
-5 & \text{if }{\small n}\text{ }{\small \equiv5,8~(13)}\text{ \ }\\
-15 & \text{if }{\small n}\text{ }{\small \equiv6,7~(13)}\text{{\small ,}}%
\end{array}
\right.  \newline{\small p}_{2}{\small =}\left\{
\begin{array}
[c]{cc}%
6 & \text{if }{\small n}\text{ }{\small \equiv1,12~(13)}\text{ \ }\\
11 & \text{if }{\small n}\text{ }{\small \equiv2,11~(13)}\text{ \ }\\
15 & \text{if }{\small n}\text{ }{\small \equiv3,10~(13)}\text{ \ }\\
18 & \text{if }{\small n}\text{ }{\small \equiv4,9~(13)}\text{ \ }\\
20 & \text{if }{\small n}\text{ }{\small \equiv5,8~(13)}\text{ \ }\\
21 & \text{if }{\small n}\text{ }{\small \equiv6,7~(13)}\text{,}%
\end{array}
\right.
\]%
\[
{\small p}_{3}{\small =}\left\{
\begin{array}
[c]{cc}%
4 & \text{if }{\small n}\text{ }{\small \equiv1,12~(13)}\text{ \ }\\
3 & \text{if }{\small n}\text{ }{\small \equiv2,11~(13)}\text{ \ }\\
-3 & \text{if }{\small n}\text{ }{\small \equiv3,10~(13)}\text{ \ }\\
-1 & \text{if }{\small n}\text{ }{\small \equiv4,9~(13)}\text{ \ }\\
-4 & \text{if }{\small n}\text{ }{\small \equiv5,8~(13)}\text{ \ }\\
-12 & \text{if }{\small n}\text{ }{\small \equiv6,7~(13)}\text{{\small ,}}%
\end{array}
\right.  {\small p}_{4}{\small =}\left\{
\begin{array}
[c]{cc}%
5 & \text{if }{\small n}\text{ }{\small \equiv1,12~(13)}\text{ \ }\\
7 & \text{if }{\small n}\text{ }{\small \equiv2,11~(13)}\text{ \ }\\
6 & \text{if }{\small n}\text{ }{\small \equiv3,10~(13)}\text{ \ }\\
2 & \text{if }{\small n}\text{ }{\small \equiv4,9~(13)}\text{ \ }\\
8 & \text{if }{\small n}\text{ }{\small \equiv5,8~(13)}\text{ \ }\\
11 & \text{if }{\small n}\text{ }{\small \equiv6,7~(13)}\text{,}%
\end{array}
\right.
\]%
\[
{\small p}_{5}{\small =}\left\{
\begin{array}
[c]{cc}%
5 & \text{if }{\small n}\text{ }{\small \equiv1,12~(13)}\text{ \ }\\
7 & \text{if }{\small n}\text{ }{\small \equiv2,11~(13)}\text{ \ }\\
6 & \text{if }{\small n}\text{ }{\small \equiv3,10~(13)}\text{ \ }\\
2 & \text{if }{\small n}\text{ }{\small \equiv4,9~(13)}\text{ \ }\\
8 & \text{if }{\small n}\text{ }{\small \equiv5,8~(13)}\text{ \ }\\
-2 & \text{if }{\small n}\text{ }{\small \equiv6,7~(13)}\text{.}%
\end{array}
\right.
\]

\end{theorem}

\begin{proof}
We prove the factorization formula (\ref{4}) by induction on $n$, using the
elliptic divisibility recurrence%
\begin{equation}
W_{n+2}W_{n-2}=W_{n+1}W_{n-1}W_{2~}^{2}-W_{3}W_{1}W_{n~}^{2}\text{,} \label{5}%
\end{equation}
where the equation (\ref{5}) is obtained from (\ref{6}) by replacing $(m,n)$
with $(n,2)$. It can be easly seen that the statement holds for $1\leq n\leq
4$. Therefore, we may assume that $n>4$. First suppose that $n$ $\equiv1~(13)$
and let the equation (\ref{4}) be true for $n+1$. Write $n=13m+1$ $(m\in%
%TCIMACRO{\U{2115} }%
%BeginExpansion
\mathbb{N}
%EndExpansion
)$. Then $n+1=13m+2$ and so $n+2=13m+3$, $n-1=13m$, $n-2=13m-1$. Thus we have
to prove that
\begin{align*}
W_{n+2}  &  =-u^{^{65m^{2}+30m+3}}v^{^{78m^{2}+36m+3}}(uv-1)^{^{52m^{2}%
+24m+3}}\\
&  \times(uv-v-1)^{^{65m^{2}+30m+3}}(v+1)^{^{-(65m^{2}+30m+3)}}.
\end{align*}
By the induction hypothesis, equation (\ref{4}) holds for $W_{n+1}$, $W_{n}$,
and $W_{n-2}$. Since $n-1\equiv0~(13)$, equation (\ref{4}) gives $W_{n-1}=0$.
Moreover, by (\ref{4})
\[
W_{1}=1\text{,}%
\]%
\[
W_{2}=uv(uv-1)(uv-v-1)(v+1)^{-1}\text{,}%
\]
and
\[
W_{3}=u^{3}v^{3}(uv-1)^{3}(uv-v-1)^{3}(v+1)^{-3}\text{.}%
\]
Furthermore%
\begin{align*}
W_{n-2}  &  =u^{^{65m^{2}-10m}}v^{78m^{2}-12m}(uv-1)^{^{52m^{2}-8m}}\\
&  \times(uv-v-1)^{^{65m^{2}-10m}}(v+1)^{^{-(65m^{2}-10m)}}\text{,}\\
W_{n+1}  &  =-u^{^{65m^{2}+20m+1}}v^{^{78m^{2}+42m+1}}(uv-1)^{^{52m^{2}%
+16m+1}}\\
&  \times(uv-v-1)^{^{65m^{2}+20m+1}}(v+1)^{^{-(65m^{2}+20m+1)}}\\
W_{n}  &  =-u^{^{65m^{2}+10m}}v^{^{78m^{2}+12m}}(uv-1)^{^{52m^{2}+8m}}\\
&  \times(uv-v-1)^{^{65m^{2}+10m}}(v+1)^{^{-(65m^{2}+10m)}},
\end{align*}
Substituting these expressions into (\ref{5}) gives
\begin{align*}
W_{n+2}  &  =-u^{^{65m^{2}+30m+3}}v^{^{78m^{2}+36m+3}}(uv-1)^{^{52m^{2}%
+24m+3}}\\
&  \times(uv-v-1)^{^{65m^{2}+30m+3}}(v+1)^{^{-(65m^{2}+30m+3)}}.
\end{align*}
This is precisely the formula predicted by (\ref{4}) for $n+2$. Hence the
induction step is complete for the case $n$ $\equiv1~(13)$. The remaining
cases can be proved in the same way.\newline
\end{proof}

As a numerical example, let $u=2$ in equation~(\ref{3}). We choose
$v=(-13+\sqrt{193})/2$ so that $v^{2}+13v-6=0$. The coefficients of the
corresponding elliptic curve in (\ref{10}) over the field $%
%TCIMACRO{\U{211a} }%
%BeginExpansion
\mathbb{Q}
%EndExpansion
(v)$ is given by
\[
b=(188v-84)/3,\qquad c=(20v-12)/9.
\]
The first terms of the associated elliptic divisibility sequence are%
\begin{align*}
W_{1}  &  =1,\\[1ex]
W_{2}  &  =(84-188v)/3,\\[1ex]
W_{3}  &  =(572317248-1282584128v)/27,\\[1ex]
W_{4}  &  =(1096821702888493056-2458018017611812864v)/2187.
\end{align*}
The subsequent terms grow rapidly. For example, formula~(\ref{4}) gives%
\[
W_{5}=-u^{10}v^{10}(uv-1)^{8}(uv-v-1)^{9}(v+1)^{-9},
\]
which evaluates to%
\[
-(525570310310338615185880645632+1177822511044877859298833596416v)/59049\text{.}%
\]
Similarly, the terms $W_{6}$, $W_{7}$ and beyond are obtained from the same
factorization formula (the expressions for $W_{6}$ and $W_{7}$ already involve
numerators with more than 90 digits and are omitted). Although this provides
an example for Theorem~3.1, it also shows that the terms of the sequence grow
very rapidly, see Subsection 3.3.

\subsection{Perfect Power Structure}

Theorem 1.1 gives a criterion for when the modular function $f_{n}$ is an
$m$-th power in $K(X_{1}(N))$. After evaluating at a point $(E,P)\in
Y_{1}(N)(K) $ yields perfect power terms in the associated elliptic
divisibility sequence.

\begin{theorem}
With notation and assumptions as in Theorem~1.1, let $f_{n}\in K(X_{1}(N))$ be
as in (\ref{1}). Suppose that%
\[
\operatorname{div}(f_{n})=\sum_{i=1}^{r}\alpha_{i}(n)\,D_{i},
\]
where each $D_{i}$ is a principal divisor on $X_{1}(N)$, and $D_{i}%
=\operatorname{div}(g_{i})$ for $g_{i}\in K(X_{1}(N))^{\times}$. If for some
integer $m\geq2$ one has $\alpha_{i}(n)\in m\mathbb{Z}$ for all $i=1,...,r$,
then there exist $h_{n}\in K(X_{1}(N))^{\times}$ and $c_{n}\in K^{\times}$
such that $f_{n}=c_{n}h_{n}^{m}$. In particular, if $c_{n}\in K^{m}$, then
$f_{n}$ is an $m$-th power in $K(X_{1}(N))^{\times}$.
\end{theorem}

\begin{proof}
By the proof of Theorem 1.1, the function $f_{n}$ can be written as a product
of modular functions $g_{i}$ with polynomials $\alpha_{i}(n)$ in $n$ of degree
at most $2$ for all $i=1,...,r$. By assumption we can write $\alpha
_{i}(n)=m\beta_{i}(n)$ for some integers $\beta_{i}(n)$. Therefore
\[
\operatorname{div}(f_{n})=\sum_{i=1}^{r}m\beta_{i}(n)\operatorname{div}%
(g_{i}).
\]
and so
\[
\operatorname{div}(f_{n})=m\operatorname{div}\left(  \prod_{i=1}^{r}%
g_{i}^{\beta_{i}(n)}\right)  .
\]
Now set
\[
h_{n}=\prod_{i=1}^{r}g_{i}^{\beta_{i}(n)}.
\]
Then $\operatorname{div}(f_{n})=\operatorname{div}(h_{n}^{m})$ and therefore
$\operatorname{div}\left(  f_{n}/h_{n}^{m}\right)  =0$. Since $X_{1}(N)$ is
complete, every rational function on $X_{1}(N)$ with trivial divisor is
constant. Hence $f_{n}/h_{n}^{m}=c_{n}$ for some $c_{n}\in K^{\times}$.
Therefore $f_{n}=c_{n}h_{n}^{m}$. If $c_{n}\in K^{m}$, say $c_{n}=d^{m}$ for
some $d\in K^{\times}$, then $f_{n}=(dh_{n})^{m}$. Hence $f_{n}$ is an $m$-th
power in $K(X_{1}(N))^{\times}$.
\end{proof}

As an immediate result of Theorem 3.2 we have the following corollary.

\begin{corollary}
With notation and assumptions as in Theorems 1.1 and~3.2. For every point
$(E,P)\in Y_{1}(N)(K)$ one has
\[
W_{n}=f_{n}(E,P)=c_{n}h_{n}^{m}(E,P).
\]
Thus, if $c_{n}\in K^{m}$, then $W_{n}$ is an $m$-th power in $K$.
\end{corollary}

\begin{proof}
Evaluating at a point $(E,P)\in Y_{1}(N)(K)$ where all functions are defined
gives
\[
W_{n}=f_{n}(E,P)=c_{n}h_{n}^{m}(E,P).
\]
If $c_{n}=d^{m}$, then
\[
W_{n}=(dh_{n}(E,P))^{m}.
\]
Thus $W_{n}$ is an $m$-th power in $K$.
\end{proof}

\subsection{Local and Global Growth Estimates}

The growth of elliptic divisibility sequences over number fields is
classically related to height theory on elliptic curves. In the nontorsion
case, logarithms of the values of division polynomials at a point are closely
related to local and canonical heights; see, for example \cite{GE2} and
\cite{JS5}. Since we work with torsion points, these nontorsion height
estimates do not apply. Instead, the estimates below follow from Theorem 1.1.

The following theorem gives a local growth estimate for the sequence
$(W_{n})_{n\geq0}$.

\begin{theorem}
With notation and assumptions as in Theorems 1.1, let $(E,P)\in Y_{1}(N)(K)$,
let $(W_{n})_{n\geq0}$ be the elliptic divisibility sequence over the number
field $K$ associated to $(E,P)$ via division polynomials $\psi_{n}$, and let
$\upsilon$ be a place of $K$. Then there exists a constant $C_{_{v}}(E,P)>0$
such that
\[
\log|W_{n}|_{\upsilon}\leq C_{\upsilon}(E,P)n^{2}%
\]
for all $n$ such that $N\nmid n$.
\end{theorem}

\begin{proof}
By Theorem~1.1, we know that the modular factors $g_{i}$ are defined and
nonzero at $(E,P)$, or equivalently, $g_{i}(E,P)$ $\in K^{\times}$ for all $i$
and that
\[
W_{n}=\prod_{i=1}^{r}g_{i}^{\alpha_{i}(n)}(E,P)
\]
for all $n$ such that $N\nmid n$. Then $|W_{n}|_{\upsilon}=\prod_{i=1}%
^{r}|g_{i}^{\alpha_{i}(n)}(E,P)|_{\upsilon}$. Since each $g_{i}(E,P)$ $\in
K^{\times}$ for all $i$, taking logarithms we obtain
\[
\log|W_{n}|_{\upsilon}=\sum_{i=1}^{r}\alpha_{i}(n)\log|g_{i}(E,P)|_{\upsilon
}\leq\sum_{i=1}^{r}|\alpha_{i}(n)||\log|g_{i}(E,P)|_{\upsilon}|\text{.}%
\]
Since each function $\alpha_{i}(n)$ has degree at most $2$ on residue classes
modulo $N$, there exists a constant $C_{i}>0$ such that $|\alpha_{i}(n)|\leq
C_{i}n^{2}$ for all $n$ such that $N\nmid n$. Therefore
\[
\log|W_{n}|_{\upsilon}\leq n^{2}\sum_{i=1}^{r}C_{i}\,|\log|g_{i}%
(E,P)|_{\upsilon}\,|\leq C_{\upsilon}(E,P)n^{2}%
\]
for some constant $C_{\upsilon}(E,P)=\max_{1\leq i\leq r}\{C_{i}|\log
|g_{i}(E,P)|_{\upsilon}\,|\}>0$.
\end{proof}

As another consequence of Theorem 1.1, we obtain a global norm estimate for
the specialized EDS terms.

\begin{corollary}
With notation and assumptions of the above theorem, define
\[
E_{n}=\left\vert \mathrm{Nm}_{K/\mathbf{%
%TCIMACRO{\U{211a} }%
%BeginExpansion
\mathbb{Q}
%EndExpansion
}}(W_{n})\right\vert
\]
for $N\nmid n$, where $\mathrm{Nm}_{K/\mathbf{%
%TCIMACRO{\U{211a} }%
%BeginExpansion
\mathbb{Q}
%EndExpansion
}}$ denotes norm for $K/\mathbf{%
%TCIMACRO{\U{211a} }%
%BeginExpansion
\mathbb{Q}
%EndExpansion
}$. Then there exists a constant $C(E,P)>0$ such that
\[
\log E_{n}\leq C(E,P)n^{2}%
\]
for all $n$ such that $N\nmid n$.
\end{corollary}

\begin{proof}
By the proof of Theorem~1.1, for $N\nmid n$ we have $f_{n}=\prod_{i=1}%
^{r}g_{i}^{\alpha_{i}(n)}$ on each residue class modulo $N$. Evaluating at
$(E,P)$ gives $W_{n}=\prod_{i=1}^{r}g_{i}^{\alpha_{i}(n)}(E,P) $. Taking the
absolute norm for $K/\mathbf{%
%TCIMACRO{\U{211a} }%
%BeginExpansion
\mathbb{Q}
%EndExpansion
}$, we obtain
\[
E_{n}=\left\vert \mathrm{Nm}_{K/\mathbf{%
%TCIMACRO{\U{211a} }%
%BeginExpansion
\mathbb{Q}
%EndExpansion
}}(W_{n})\right\vert =\prod_{i=1}^{r}\left\vert \mathrm{Nm}_{K/\mathbf{%
%TCIMACRO{\U{211a} }%
%BeginExpansion
\mathbb{Q}
%EndExpansion
}}(g_{i}(E,P))\right\vert ^{\alpha_{i}(n)}.
\]
Therefore
\[
\log E_{n}=\sum_{i=1}^{r}\alpha_{i}(n)\log\left\vert \mathrm{Nm}_{K/\mathbf{%
%TCIMACRO{\U{211a} }%
%BeginExpansion
\mathbb{Q}
%EndExpansion
}}(g_{i}(E,P))\right\vert \leq\sum_{i=1}^{r}|\alpha_{i}(n)\,|\,|\log\left\vert
\mathrm{Nm}_{K/\mathbf{%
%TCIMACRO{\U{211a} }%
%BeginExpansion
\mathbb{Q}
%EndExpansion
}}(g_{i}(E,P))\right\vert \,|.
\]
Since each function $\alpha_{i}(n)$ has degree at most $2$ on residue classes
modulo $N$, there exists a constant $C_{i}>0$ such that $|\alpha_{i}(n)|\leq
C_{i}n^{2}$ for all $n$ such that $N\nmid n$. Since the functions $\alpha
_{i}(n)$ have degree at most $2$ on residue classes modulo $N$, there is a
constant $C(E,P)>0$ such that
\[
\log E_{n}\leq C(E,P)n^{2}.
\]

\end{proof}

\section{Lifting and Norm Descent}

Let $(E,C)$ be a point of $Y_{0}(N)(K)$, where $C$ is a cyclic subgroup of
order $N$ and choose a generator $P\in C$, yielding a lift $(E,P)\in
Y_{1}(N)(K)$ and let $f_{n}$ be as in (\ref{1}). In this section, we establish
that evaluating the modular function $f_{n}$ at this lifted point, i.e.,
$W_{n}=f_{n}(E,P)$ defines an EDS. Furthermore we demonstrade that the
resulting factorization structure descents to $X_{0}(N)$ via the norm map.

\begin{proof}
[Proof of Theorem 1.2]Let $(E,C)\in Y_{0}(N)(K)$ be given. Since $C$ is a
cyclic subgroup of order $N$, any generator $P\in C$ determines a point
$(E,P)\in Y_{1}(N)(K)$, which is a lift of $(E,C)$ along the morphism
\[
\phi:Y_{1}(N)\rightarrow Y_{0}(N)\text{, }(E,P)\mapsto(E,\langle P\rangle)
\]
obtained by forgetting the generator, but keeping the group it generates; see
\cite[p.40]{DS}. By (\ref{1}), the functions
\[
f_{n}=\mathcal{P}^{\ast}(\psi_{n})
\]
are modular functions on $Y_{1}(N)$. Now evaluating at the lifted point
$(E,P)$ gives
\[
W_{n}=f_{n}(E,P)=\psi_{n}(P),
\]
which is the elliptic divisibility sequence associated to $(E,P)$. By the
proof of Theorem~1.1, the $f_{n}$ can be factored as
\[
f_{n}=\prod_{i=1}^{r}g_{i}^{\alpha_{i}(n)}\text{,}%
\]
and evaluating this at $(E,P)$ gives
\[
W_{n}=\prod_{i=1}^{r}g_{i}^{\alpha_{i}(n)}(E,P)\text{,}%
\]
where $g_{i}\in K(X_{1}(N))^{\times}$ and each $\alpha_{i}(n)$ is a polynomial
in $n$ of degree at most $2$ taking integer values on suitable residue classes
$n$ modulo $N$. Now let $kP$ be another generator of $C$, with $k\in
(\mathbb{Z}/N\mathbb{Z})^{\times}$. Then $(E,kP)$ is the image of $(E,P)$
under the natural action of $(\mathbb{Z}/N\mathbb{Z})^{\times}$ on $Y_{1}(N)$.
Since this action permutes the factors in $W_{n}$, the factorization structure
is preserved in the sense that the factors are transformed by the induced
action while the exponent pattern remains unchanged.
\end{proof}

We are now ready to prove Theorem 1.3.

\begin{proof}
[Proof of theorem 1.3]By the proofs of Theorems~1.1 and 1.2, the functions
$f_{n}$ are defined on $X_{1}(N)$ and their evaluation at $(E,P)$ is
compatible with the natural action of $(\mathbb{Z}/N\mathbb{Z})^{\times}$,
respectively. Since the morphism
\[
\phi:X_{1}(N)\rightarrow X_{0}(N)
\]
is a finite covering, the norm map
\[
\mathrm{Nm}_{K(X_{1}(N))/K(X_{0}(N))}:K(X_{1}(N))^{\times}\rightarrow
K(X_{0}(N))^{\times}%
\]
is well-defined, thus $F_{n}=\mathrm{Nm}_{K(X_{1}(N))/K(X_{0}(N))}(f_{n})\in
K(X_{0}(N))^{\times}$.

On the other hand, the function field extension $K(X_{1}(N))/K(X_{0}(N))$ is a
Galois extension with Galois group $(\mathbb{Z}/N\mathbb{Z})^{\times}$, where
the group acting on sections via $\mathcal{P}\mapsto k\mathcal{P}$,
$k\in(\mathbb{Z}/N\mathbb{Z})^{\times}$; see \cite[Section 7.5.]{DS}. Hence,
the function $F_{n}$ is defined as the norm of the $f_{n}$ over the group
$(\mathbb{Z}/N\mathbb{Z})^{\times}$, which is given by the product of
pull-backs along the sections $k\mathcal{P}$ as
\[
F_{n}=\prod_{k\in(\mathbb{Z}/N\mathbb{Z})^{\times}}(k\mathcal{P)}^{\ast}%
(\psi_{n})\text{,}%
\]
(see; \cite[Section 1.8]{KM}, for more details). Evaluating at $(E,C)$ yields
\[
F_{n}(E,C)=\prod_{k\in(\mathbb{Z}/N\mathbb{Z})^{\times}}f_{n}(E,kP).
\]

The independence of this factorization from the choice of generator $P$
follows from the proof of Theorem~1.2. Finally, applying the norm to the
factorization given in the proof of Theorem 1.1,
\[
f_{n}=\prod_{i=1}^{r}g_{i}^{\alpha_{i}(n)}%
\]
and using multiplicativity, we obtain
\[
F_{n}=\prod_{i=1}^{r}(\mathrm{Nm}_{K(X_{1}(N))/K(X_{0}(N))}(g_{i}%
))^{\alpha_{i}(n)}=\prod_{i=1}^{r}G_{i}^{\alpha_{i}(n)}\text{,}%
\]
where $G_{i}=\mathrm{Nm}_{K(X_{1}(N))/K(X_{0}(N))}(g_{i})$.
\end{proof}

\section{Appendix}

In this appendix, we give the explicit factorization formulas over a number
field $K$ for the levels $N\in\{11,13,14,15,16,18\}$. These cases are
naturally motivated by the arithmetic of torsion points over quadratic fields.
On the other hand, the levels $N$ $\in\{2,3,...,9,10,12\}$ were studied in
earlier works over $%
%TCIMACRO{\U{211a} }%
%BeginExpansion
\mathbb{Q}
%EndExpansion
$ in \cite{BG,BO1} and remain valid over any number field by base change. The
additional levels treated here correspond to the quadratic torsion
classification of Kenku and Momose \cite{KMO}. We give only the formulas for
the elliptic divisibility sequences associated to modular curves $Y_{1}(N)$ in
this appendix. A detailed proof is given in Section 3 for the case $N=13$. The
remaning levels can be proved in the same way, by using the induction and
elliptic divisibility sequence recurrence (\ref{6}).

For the remainder of the paper, we fix the models for $Y_{1}(N)$ given in
\cite[p. 1144]{SU} with coordinates $(u,v)$. For each level $N\in
\{11,13,14,15,16,18\}$, let $t_{(u,v)}\in Y_{1}(N)(K)$ be a $K$-rational point
on the chosen affine model where $K$ is a number field. Associated to $t$\ is
a pair $(E_{t},P_{t})$ where $E_{t}$ is the curve over $K(t)$ given by the
equation%
\[
E_{t}:y^{2}+(1-c)xy-by=x^{3}-bx^{2}%
\]
with%
\[
c=s(r-1),b=cr
\]
where $P_{t}=(0,0)$ is a point of exact order $N$, and $r$ and $s$ are given
functions of $u$ and $v$; see \cite[pp. 1144-1145]{SU}, for models and
birational maps. Evaluating the functions $f_{n}$ at $t$ recovers the
associated elliptic divisibility sequence
\[
W_{n}(t)=W_{n}(u,v)=f_{n}(E_{t},P_{t})\text{.}%
\]
We then express the functions $f_{n}$ explicitly in terms of modular
coordinates on $Y_{1}(N)$.%
\[
\]

\textbf{1. }$N=11$\textbf{.} Let $Y_{1}(11)$ be the affine model
\begin{equation}
v^{2}+(u^{2}+1)v+u=0
\end{equation}
with
\[
r=1+uv,s=1-u\text{.}%
\]
Then%
\begin{equation}
W_{n}(t)=\left\{
\begin{array}
[c]{cc}%
{\small \varepsilon}\prod_{i=1}^{4}g_{i}^{\alpha_{i}(n)}(t) & \text{if
}11\nmid n\\
0 & \text{if }11|n\text{,}%
\end{array}
\right.
\end{equation}
where
\[
g_{1}(t)=u,g_{2}(t)=v,g_{3}(t)=u-1,g_{4}(y)=v+1
\]
and
\begin{align*}
\alpha_{1}(n)  &  =(n^{2}-p_{1})/11,\alpha_{2}(n)=(8n^{2}-p_{2})/11,\\
\alpha_{3}(n)  &  =(4n^{2}-p_{3})/11,\alpha_{4}(n)=(4n^{2}-p_{4})/11
\end{align*}
with%
\[
{\small \varepsilon=}\left\{
\begin{array}
[c]{ll}%
{\small +1} & \text{if }{\small n}\text{ }{\small \equiv
1,10,13,14,15,16,17,18,19,20~(22)}\text{ \ \ }{\small \ }\\
{\small -1} & \text{if }{\small n}\text{ }{\small \equiv
2,3,4,5,6,7,8,9,12,21~(22)}\text{,}%
\end{array}
\right.
\]%
\[
{\small p}_{1}{\small =}\left\{
\begin{array}
[c]{cc}%
1 & \text{if }{\small n}\text{ }{\small \equiv1,10~(11)}\text{ \ }\\
4 & \text{if }{\small n}\text{ }{\small \equiv2,9~(11)}\text{ \ }\\
9 & \text{if }{\small n}\text{ }{\small \equiv3,8~(11)}\text{ \ }\\
5 & \text{if }{\small n}\text{ }{\small \equiv4,7~(11)}\text{ \ }\\
3 & \text{if }{\small n}\text{ }{\small \equiv5,6~(11)}\text{, \ }%
\end{array}
\right.  \newline{\small p}_{2}{\small =}\left\{
\begin{array}
[c]{cc}%
8 & \text{if }{\small n}\text{ }{\small \equiv1,10~(11)}\text{ \ }\\
10 & \text{if }{\small n}\text{ }{\small \equiv2,9~(11)}\text{ \ }\\
6 & \text{if }{\small n}\text{ }{\small \equiv3,8~(11)}\text{ \ }\\
7 & \text{if }{\small n}\text{ }{\small \equiv4,7~(11)}\text{ \ }\\
2 & \text{if }{\small n}\text{ }{\small \equiv5,6~(11)}\text{, \ }%
\end{array}
\right.
\]%
\[
{\small p}_{3}{\small =}\left\{
\begin{array}
[c]{cc}%
4 & \text{if }{\small n}\text{ }{\small \equiv1,10~(11)}\text{ \ }\\
5 & \text{if }{\small n}\text{ }{\small \equiv2,9~(11)}\text{ \ }\\
3 & \text{if }{\small n}\text{ }{\small \equiv3,8~(11)}\text{ \ }\\
-2 & \text{if }{\small n}\text{ }{\small \equiv4,7~(11)}\text{ \ }\\
1 & \text{if }{\small n}\text{ }{\small \equiv5,6~(11)}\text{, \ }%
\end{array}
\right.  {\small p}_{4}{\small =}\left\{
\begin{array}
[c]{cc}%
4 & \text{if }{\small n}\text{ }{\small \equiv1,10~(11)}\text{ \ }\\
5 & \text{if }{\small n}\text{ }{\small \equiv2,9~(11)}\text{ \ }\\
3 & \text{if }{\small n}\text{ }{\small \equiv3,8~(11)}\text{ \ }\\
9 & \text{if }{\small n}\text{ }{\small \equiv4,7~(11)}\text{ \ }\\
12 & \text{if }{\small n}\text{ }{\small \equiv5,6~(11)}\text{. \ }%
\end{array}
\right.
\]%
\[
\]

\textbf{2.} $N=14$\textbf{.} Let $Y_{1}(14)$ be the affine model%
\begin{equation}
v^{2}+(u^{2}+u)v+u=0
\end{equation}
and
\[
r=1-(u+v)/((v+1)(u+v+1)),s=(1-u)/(v+1).
\]
Then%
\begin{equation}
W_{n}(t)=\left\{
\begin{array}
[c]{cc}%
{\small \varepsilon}\prod_{i=1}^{6}g_{i}^{\alpha_{i}(n)}(t) & \text{if
}14\nmid n\\
0 & \text{if }14|n\text{,}%
\end{array}
\right.
\end{equation}
where
\[
g_{1}(t)=u-1,g_{2}(t)=v,g_{3}(t)=u+v,g_{4}(t)=v+1,
\]%
\[
g_{5}(t)=u+v+1,g_{6}(t)=v^{2}+uv+v+1,\newline%
\]
and
\[
\alpha_{1}(n)=(11n^{2}-p_{1})/28,\alpha_{2}(n)=(n^{2}-p_{2})/28,\alpha
_{3}(n)=(3n^{2}-p_{3})/7,
\]%
\[
\alpha_{4}(n)=-(8n^{2}-p_{4})/7,\alpha_{5}(n)=-(3n^{2}-p_{5})/4,\alpha
_{6}(n)=(9n^{2}-p_{6})/28,
\]
with%
\[
{\small \varepsilon=}\left\{
\begin{array}
[c]{ll}%
{\small +1} & \text{if }{\small n}\text{ }{\small \equiv
1,4,7,8,12,13,17,18,19,22,23,25,26~(28)}\text{ \ \ }{\small \ }\\
{\small -1} & \text{if }{\small n}\text{ }{\small \equiv
2,3,5,6,9,10,11,15,16,20,21,24,27~(28)},
\end{array}
\right.
\]%
\[
{\small p}_{1}{\small =}\left\{
\begin{array}
[c]{cc}%
11 & \text{if }{\small n}\text{ }{\small \equiv1,13~(14)}\\
16 & \text{if }{\small n}\text{ }{\small \equiv2,12~(14)}\\
15 & \text{if }{\small n}\text{ }{\small \equiv3,11~(14)}\\
8 & \text{if }{\small n}\text{ }{\small \equiv4,10~(14)}\\
23 & \text{if }{\small n}\text{ }{\small \equiv5,9~(14)}\\
32 & \text{if }{\small n}\text{ }{\small \equiv6,8~(14)}\\
35 & \text{if }{\small n}\text{ }{\small \equiv7~(14)},
\end{array}
\right.  \newline{\small p}_{2}{\small =}\left\{
\begin{array}
[c]{cc}%
1 & \text{if }{\small n}\text{ }{\small \equiv1,13~(14)}\\
4 & \text{if }{\small n}\text{ }{\small \equiv2,12~(14)}\\
9 & \text{if }{\small n}\text{ }{\small \equiv3,11~(14)}\\
16 & \text{if }{\small n}\text{ }{\small \equiv4,10~(14)}\\
25 & \text{if }{\small n}\text{ }{\small \equiv5,9~(14)}\\
36 & \text{if }{\small n}\text{ }{\small \equiv6,8~(14)}\\
49 & \text{if }{\small n}\text{ }{\small \equiv7~(14)},
\end{array}
\right.
\]%
\[
{\small p}_{3}{\small =}\left\{
\begin{array}
[c]{cc}%
3 & \text{if }{\small n}\text{ }{\small \equiv1,6,8,13~(14)}\\
5 & \text{if }{\small n}\text{ }{\small \equiv2,5,9,12~(14)}\\
6 & \text{if }{\small n}\text{ }{\small \equiv3,4,10,11~(14)}\\
0 & \text{if }{\small n}\text{ }{\small \equiv7~(14)},
\end{array}
\right.  {\small p}_{4}{\small =}\left\{
\begin{array}
[c]{cc}%
8 & \text{if }{\small n}\text{ }{\small \equiv1,6,8,13~(14)}\\
11 & \text{if }{\small n}\text{ }{\small \equiv2,5,9,12~(14)}\\
9 & \text{if }{\small n}\text{ }{\small \equiv3,4,10,11~(14)}\\
14 & \text{if }{\small n}\text{ }{\small \equiv7~(14)},
\end{array}
\right.
\]%
\[
{\small p}_{5}{\small =}\left\{
\begin{array}
[c]{cc}%
3 & \text{if }{\small n}\text{ }{\small \equiv1,3,5,7,9,11,13~(14)}\\
4 & \text{if }{\small n}\text{ }{\small \equiv2,4,6,8,10,12~(14)},
\end{array}
\right.  {\small p}_{6}{\small =}\left\{
\begin{array}
[c]{cc}%
9 & \text{if }{\small n}\text{ }{\small \equiv1,13~(14)}\\
8 & \text{if }{\small n}\text{ }{\small \equiv2,12~(14)}\\
-3 & \text{if }{\small n}\text{ }{\small \equiv3,11~(14)}\\
4 & \text{if }{\small n}\text{ }{\small \equiv4,10~(14)}\\
1 & \text{if }{\small n}\text{ }{\small \equiv5,9~(14)}\\
-12 & \text{if }{\small n}\text{ }{\small \equiv6,8~(14)}\\
-7 & \text{if }{\small n}\text{ }{\small \equiv7~(14)}.
\end{array}
\right.
\]%
\[
\]

\textbf{3.} $N=15$\textbf{. }\ Let $Y_{1}(15)$ be the affine model
\begin{equation}
v^{2}+(u^{2}+u+1)v+u^{2}=0
\end{equation}
and
\[
r=1+(uv+v^{2})/(u^{3}+u^{2}v+u^{2}),s=1+v/(u^{2}+u)\text{.}%
\]
Then%
\begin{equation}
W_{n}(t)=\left\{
\begin{array}
[c]{cc}%
{\small \varepsilon}\prod_{i=1}^{6}g_{i}^{\alpha_{i}(n)}(t) & \text{if
}15\nmid n\\
0 & \text{if }15|n\text{,}%
\end{array}
\right.
\end{equation}
where
\[
g_{1}(t)=u,g_{2}(t)=v,g_{3}(t)=u+1,
\]%
\[
g_{4}(t)=u+v,g_{5}(t)=u^{3}-v,g_{6}(t)=u+v+1,
\]
and
\[
\alpha_{1}(n)=-(11n^{2}-p_{1})/15,\alpha_{2}(n)=(n^{2}-p_{2})/15,\alpha
_{3}(n)=-(2n^{2}-p_{3})/5,
\]%
\[
\alpha_{4}(n)=(2n^{2}-p_{4})/5,\alpha_{5}(n)=(n^{2}-p_{5})/3,\alpha
_{6}(n)=-(17n^{2}-p_{6})/15,
\]

with%
\[
{\small \varepsilon=}\left\{
\begin{array}
[c]{ll}%
{\small +1} & \text{if }{\small n}\text{ }{\small \equiv1,2,3,4,6,8,10~(15)}%
\text{ \ \ }{\small \ }\\
{\small -1} & \text{if }{\small n}\text{ }{\small \equiv
5,7,9,11,12,13,14~(15)},
\end{array}
\right.
\]%
\[
{\small p}_{1}{\small =}\left\{
\begin{array}
[c]{cc}%
11 & \text{if }{\small n}\text{ }{\small \equiv1,14~(15)}\\
14 & \text{if }{\small n}\text{ }{\small \equiv2,13~(15)}\\
9 & \text{if }{\small n}\text{ }{\small \equiv3,12~(15)}\\
26 & \text{if }{\small n}\text{ }{\small \equiv4,11~(15)}\\
5 & \text{if }{\small n}\text{ }{\small \equiv5,10~(15)}\\
-9 & \text{if }{\small n}\text{ }{\small \equiv6,9~(15)}\\
-1 & \text{if }{\small n}\text{ }{\small \equiv7,8~(15)},
\end{array}
\right.  \newline{\small p}_{2}{\small =}\left\{
\begin{array}
[c]{cc}%
1 & \text{if }{\small n}\text{ }{\small \equiv1,14~(15)}\\
4 & \text{if }{\small n}\text{ }{\small \equiv2,7,8,13~(15)}\\
9 & \text{if }{\small n}\text{ }{\small \equiv3,12~(15)}\\
16 & \text{if }{\small n}\text{ }{\small \equiv4,11~(15)}\\
10 & \text{if }{\small n}\text{ }{\small \equiv5,10~(15)}\\
6 & \text{if }{\small n}\text{ }{\small \equiv6,9~(15)},
\end{array}
\right.
\]%
\[
{\small p}_{3}{\small =}\left\{
\begin{array}
[c]{cc}%
2 & \text{if }{\small n}\text{ }{\small \equiv1,4,6,9,11,14~(15)}\\
3 & \text{if }{\small n}\text{ }{\small \equiv2,3,7,8,12,13~(15)}\\
5 & \text{if }{\small n}\text{ }{\small \equiv5,10~(15),}%
\end{array}
\right.  {\small p}_{4}{\small =}\left\{
\begin{array}
[c]{cc}%
2 & \text{if }{\small n}\text{ }{\small \equiv1,4,6,9,11,14~(15)}\\
3 & \text{if }{\small n}\text{ }{\small \equiv2,3,7,8,12,13~(15)}\\
0 & \text{if }{\small n}\text{ }{\small \equiv5,10~(15),}%
\end{array}
\right.
\]%
\[
{\small p}_{5}{\small =}\left\{
\begin{array}
[c]{cc}%
1 & \text{if }{\small n}\text{ }{\small \equiv1,2,4,5,7,8,10,11,13,14~(15)}\\
0 & \text{if }{\small n}\text{ }{\small \equiv3,6,9,12~(15)},
\end{array}
\right.  {\small p}_{6}{\small =}\left\{
\begin{array}
[c]{cc}%
17 & \text{if }{\small n}\text{ }{\small \equiv1,4,11,14~(15)}\\
23 & \text{if }{\small n}\text{ }{\small \equiv2,7,8,13~(15)}\\
18 & \text{if }{\small n}\text{ }{\small \equiv3,12~(15)}\\
20 & \text{if }{\small n}\text{ }{\small \equiv5,10~(15)}\\
27 & \text{if }{\small n}\text{ }{\small \equiv6,9~(15)}.
\end{array}
\right.
\]%
\[
\]

\textbf{4. }$N=16$. Let $Y_{1}(16)$ be the affine model
\begin{equation}
v^{2}+(u^{3}+u^{2}-u+1)v+u^{2}=0
\end{equation}
with
\[
r=(u^{2}-uv+v^{2}+v)/(u^{2}+u-v-1),s=(u-v)/(u+1)\text{.}%
\]
Then%
\begin{equation}
W_{n}(t)=\left\{
\begin{array}
[c]{cc}%
{\small \varepsilon}\prod_{i=1}^{7}g_{i}^{\alpha_{i}(n)}(t) & \text{if
}16\nmid n\\
0 & \text{if }16|n\text{,}%
\end{array}
\right.
\end{equation}
where
\[
g_{1}(t)=u,g_{2}(t)=v,g_{3}(t)=u+1,g_{4}(t)=u-v,
\]%
\[
g_{5}(t)=v+1,g_{6}(t)=u-v-1,g_{7}(t)=u^{2}+u-v-1,
\]
\newline and
\[
\alpha_{1}(n)=(21n^{2}-p_{1})/32,\alpha_{2}(n)=(11n^{2}-p_{2})/32,\alpha
_{3}(n)=-(n^{2}-p_{3})/16,
\]%
\[
\alpha_{4}(n)=(3n^{2}-p_{4})/8,\alpha_{5}(n)=(7n^{2}-p_{5})/16,
\]%
\[
\alpha_{6}(n)=(13n^{2}-p_{6})/32,\alpha_{7}(n)=-(3n^{2}-p_{7})/4,
\]
with%
\[
{\small \varepsilon=}\left\{
\begin{array}
[c]{ll}%
{\small +1} & \text{if }{\small n}\text{ }{\small \equiv
1,5,6,7,10,13,17,18,20,21,23,24,28,29,30~(32)}\text{ \ \ }{\small \ }\\
{\small -1} & \text{if }{\small n}\text{ }{\small \equiv
2,3,4,8,9,11,12,14,15,19,22,25,26,27,31~(32)},
\end{array}
\right.
\]%
\[
{\small p}_{1}{\small =}\left\{
\begin{array}
[c]{cc}%
21 & \text{if }{\small n}\text{ }{\small \equiv1,15~(16)}\\
20 & \text{if }{\small n}\text{ }{\small \equiv2,14~(16)}\\
-3 & \text{if }{\small n}\text{ }{\small \equiv3,13~(16)}\\
16 & \text{if }{\small n}\text{ }{\small \equiv4,12~(16)}\\
13 & \text{if }{\small n}\text{ }{\small \equiv5,11~(16)}\\
-12 & \text{if }{\small n}\text{ }{\small \equiv6,10~(16)}\\
-27 & \text{if }{\small n}\text{ }{\small \equiv7,9~(16)}\\
0 & \text{if }{\small n}\text{ }{\small \equiv8~(16)},
\end{array}
\right.  \newline{\small p}_{2}{\small =}\left\{
\begin{array}
[c]{cc}%
11 & \text{if }{\small n}\text{ }{\small \equiv1,15~(16)}\\
12 & \text{if }{\small n}\text{ }{\small \equiv2,6,10,14~(16)}\\
3 & \text{if }{\small n}\text{ }{\small \equiv3,13~(16)}\\
16 & \text{if }{\small n}\text{ }{\small \equiv4,12~(16)}\\
19 & \text{if }{\small n}\text{ }{\small \equiv5,11~(16)}\\
27 & \text{if }{\small n}\text{ }{\small \equiv7,9~(16)}\\
32 & \text{if }{\small n}\text{ }{\small \equiv8~(16)},
\end{array}
\right.
\]%
\[
{\small p}_{3}{\small =}\left\{
\begin{array}
[c]{cc}%
1 & \text{if }{\small n}\text{ }{\small \equiv1,7,9,15~(16)}\\
4 & \text{if }{\small n}\text{ }{\small \equiv2,6,10,14~(16)}\\
9 & \text{if }{\small n}\text{ }{\small \equiv3,5,11,13~(16)}\\
0 & \text{if }{\small n}\text{ }{\small \equiv4,8,12~(16)},
\end{array}
\right.  {\small p}_{4}{\small =}\left\{
\begin{array}
[c]{cc}%
3 & \text{if }{\small n}\text{ }{\small \equiv1,3,5,7,9,11,13,15~(16)}\\
4 & \text{if }{\small n}\text{ }{\small \equiv2,6,10,14~(16)}\\
0 & \text{if }{\small n}\text{ }{\small \equiv4,8,12~(16),}%
\end{array}
\right.
\]%
\[
{\small p}_{5}{\small =}\left\{
\begin{array}
[c]{cc}%
7 & \text{if }{\small n}\text{ }{\small \equiv1,7,9,15~(16)}\\
12 & \text{if }{\small n}\text{ }{\small \equiv2,6,10,14~(16)}\\
15 & \text{if }{\small n}\text{ }{\small \equiv3,5,11,13~(16)}\\
16 & \text{if }{\small n}\text{ }{\small \equiv4,12~(16)}\\
0 & \text{if }{\small n}\text{ }{\small \equiv8~(16)},
\end{array}
\right.  {\small p}_{6}{\small =}\left\{
\begin{array}
[c]{cc}%
13 & \text{if }{\small n}\text{ }{\small \equiv1,15~(16)}\\
20 & \text{if }{\small n}\text{ }{\small \equiv2,6,10,14~(16)}\\
21 & \text{if }{\small n}\text{ }{\small \equiv3,13~(16)}\\
16 & \text{if }{\small n}\text{ }{\small \equiv4,12~(16)}\\
5 & \text{if }{\small n}\text{ }{\small \equiv5,11~(16)}\\
29 & \text{if }{\small n}\text{ }{\small \equiv7,9~(16)}\\
32 & \text{if }{\small n}\text{ }{\small \equiv8~(16)},
\end{array}
\right.
\]%
\[
{\small p}_{7}{\small =}\left\{
\begin{array}
[c]{cc}%
3 & \text{if }{\small n}\text{ }{\small \equiv1,3,5,7,9,11,13,15~(16)}\\
4 & \text{if }{\small n}\text{ }{\small \equiv2,4,6,8,10,12,14~(16)}.
\end{array}
\right.
\]%
\[
\]

\textbf{5.} $N=18$\textbf{.} Let $Y_{1}(18)$ be the affine model
\begin{equation}
v^{2}+(u^{3}-2u^{2}+3u+1)v+2u=0
\end{equation}
with
\[
r=(u^{2}-uv-3u+1)/((u-1)^{2}(uv+1)),s=(u^{2}-2u-v)/(u^{2}-uv-3u-v^{2}%
-2v)\text{.}%
\]
Then%
\begin{equation}
W_{n}(t)=\left\{
\begin{array}
[c]{cc}%
{\small \varepsilon}\prod_{i=1}^{8}g_{i}^{\alpha_{i}(n)}(t) & \text{if
}18\nmid n\\
0 & \text{if }18|n\text{,}%
\end{array}
\right.
\end{equation}
where
\[
g_{1}(t)=u,g_{2}(t)=v,g_{3}(t)=u+v,g_{4}(t)=v+1,g_{5}(t)=u-1,
\]%
\[
g_{6}(t)=u^{2}v+u-uv+v,g_{7}(t)=u^{2}-uv-3u+1,g_{8}(t)=uv+1
\]
\newline and
\[
\alpha_{1}(n)=(7n^{2}-p_{1})/18,\alpha_{2}(n)=p_{2},\alpha_{3}(n)=(5n^{2}%
-p_{3})/12,
\]%
\[
\alpha_{4}(n)=(n^{2}-p_{4})/12,\alpha_{5}(n)=-(17n^{2}-p_{5})/9,\alpha
_{6}(n)=-(n^{2}-p_{6})/36,
\]%
\[
\alpha_{7}(n)=(n^{2}-p_{7})/3,\alpha_{8}(n)=-(3n^{2}-p_{8})/4,
\]
with%
\[
{\small \varepsilon=}\left\{
\begin{array}
[c]{ll}%
{\small +1} & \text{if }{\small n}\text{ }{\small \equiv
1,4,5,8,10,13,14,17,20,21,24,25,27,29,30,33,34~(36)}\text{ \ \ }{\small \ }\\
{\small -1} & \text{if }{\small n}\text{ }{\small \equiv
2,3,6,7,9,11,12,15,16,19,22,23,26,28,31,32,35~(36)},
\end{array}
\right.
\]%
\[
{\small p}_{1}{\small =}\left\{
\begin{array}
[c]{cc}%
7 & \text{if }{\small n}\text{ }{\small \equiv1,17~(18)}\\
10 & \text{if }{\small n}\text{ }{\small \equiv2,16~(18)}\\
9 & \text{if }{\small n}\text{ }{\small \equiv3,9,15~(18)}\\
4 & \text{if }{\small n}\text{ }{\small \equiv4,14~(18)}\\
13 & \text{if }{\small n}\text{ }{\small \equiv5,13~(18)}\\
18 & \text{if }{\small n}\text{ }{\small \equiv6,12~(18)}\\
19 & \text{if }{\small n}\text{ }{\small \equiv7,11~(18)}\\
16 & \text{if }{\small n}\text{ }{\small \equiv8,10~(18)},
\end{array}
\right.  \newline{\small p}_{2}{\small =}\left\{
\begin{array}
[c]{cc}%
1 & \text{if }{\small n}\text{ }{\small \equiv9~(18)}\\
0 & \text{otherwise }%
\end{array}
\right.
\]%
\[
{\small p}_{3}{\small =}\left\{
\begin{array}
[c]{cc}%
5 & \text{if }{\small n}\text{ }{\small \equiv1,5,7,11,13,17~(18)}\\
8 & \text{if }{\small n}\text{ }{\small \equiv2,4,8,10,14,16~(18)}\\
9 & \text{if }{\small n}\text{ }{\small \equiv3,9,15~(18)}\\
0 & \text{if }{\small n}\text{ }{\small \equiv6,12~(18)},
\end{array}
\right.  {\small p}_{4}{\small =}\left\{
\begin{array}
[c]{cc}%
1 & \text{if }{\small n}\text{ }{\small \equiv1,17~(18)}\\
4 & \text{if }{\small n}\text{ }{\small \equiv2,16~(18)}\\
9 & \text{if }{\small n}\text{ }{\small \equiv3,15~(18)}\\
16 & \text{if }{\small n}\text{ }{\small \equiv4,8,10,14~(18)}\\
13 & \text{if }{\small n}\text{ }{\small \equiv5,7,11,13~(18)}\\
12 & \text{if }{\small n}\text{ }{\small \equiv6,12~(18)}\\
21 & \text{if }{\small n}\text{ }{\small \equiv9~(18)},
\end{array}
\right.
\]%
\[
{\small p}_{5}{\small =}\left\{
\begin{array}
[c]{cc}%
17 & \text{if }{\small n}\text{ }{\small \equiv1,8,10,17~(18)}\\
23 & \text{if }{\small n}\text{ }{\small \equiv2,7,11,16~(18)}\\
18 & \text{if }{\small n}\text{ }{\small \equiv3,6,12,15~(18)}\\
20 & \text{if }{\small n}\text{ }{\small \equiv4,5,13,14~(18)}\\
27 & \text{if }{\small n}\text{ }{\small \equiv9~(18)},
\end{array}
\right.  {\small p}_{6}{\small =}\left\{
\begin{array}
[c]{cc}%
1 & \text{if }{\small n}\text{ }{\small \equiv1,17~(18)}\\
4 & \text{if }{\small n}\text{ }{\small \equiv2,16~(18)}\\
9 & \text{if }{\small n}\text{ }{\small \equiv3,9,15~(18)}\\
16 & \text{if }{\small n}\text{ }{\small \equiv4,14~(18)}\\
25 & \text{if }{\small n}\text{ }{\small \equiv5,13~(18)}\\
0 & \text{if }{\small n}\text{ }{\small \equiv6,12~(18)}\\
13 & \text{if }{\small n}\text{ }{\small \equiv7,11~(18)}\\
28 & \text{if }{\small n}\text{ }{\small \equiv8,10~(18)},
\end{array}
\right.
\]%
\begin{align*}
{\small p}_{7}  &  {\small =}\left\{
\begin{array}
[c]{cc}%
1 & \text{if }{\small n}\text{ }{\small \equiv
1,2,4,5,7,8,10,11,13,14,16,17~(18)}\\
0 & \text{if }{\small n}\text{ }{\small \equiv3,6,9,12,15~(18)},
\end{array}
\right. \\
{\small p}_{8}  &  {\small =}\left\{
\begin{array}
[c]{cc}%
3 & \text{if }{\small n}\text{ }{\small \equiv1,3,5,7,9,11,13,15,17~(18)}\\
4 & \text{if }{\small n}\text{ }{\small \equiv2,4,6,8,10,12,14,16~(18)}.
\end{array}
\right.
\end{align*}

\textbf{Funding. }The second author is supported by the Scientific and
Technological Research Council of T\"{u}rkiye (T\"{U}B\.{I}TAK)-2211.


\begin{thebibliography}{99}                                                                                               %


\bibitem {CH}J. Cheon, S. Hahn, Explicit valuations of division polynomials of
an elliptic curve, Manuscripta Math. \textbf{97} (1998), 319-328.

\bibitem {CR}G. Cornelissen, J. Reynolds, The perfect power problem for
elliptic curves over function fields, N.Y. J. Math. 22 (2016) 95--114

\bibitem {DS}F. Diamond F. and J. Shurman, A First Course in Modular Forms,
Graduate Texts in Mathematics 228, Springer-Verlag, New York 2005.

\bibitem {GE1}G. Everest, A. van der Poorten, I. Shparlinski, T. Ward,
Recurrence Sequences, Mathematical Surveys and Monographs 104, AMS,
Providence, RI, 2003.

\bibitem {GE2}G. Everest, A. and T. Ward, The canonical height of an algebraic
point on an elliptic curve, New York J. Math. \textbf{6} (2000) 331-342.

\bibitem {BG}B. Gezer, \textit{Elliptic divisibility sequences}%
,\textit{\ squares and cubes}, Publ. Math. Debrecen \textbf{83} (3), 481-515, 2013.

\bibitem {BO1}B. Gezer, O. Bizim, \textit{Squares in elliptic divisibility
sequences}, Acta Arithmetica \textbf{144} (2), 125-134, 2010.

\bibitem {IP}P. Ingram, V. Mahe, J. H. Silverman, K. E. Stange and M. Streng.
Algebraic divisibility sequences over function fields. J. Austr. Math. Soc.
\textbf{92 }(2012), 99-126.

\bibitem {KM}N. M. Katz and B. Mazur, Arithmetic Moduli of Elliptic Curves,
Ann. of Math. Stud., 108. Princeton Univ. Press, New Jersey, 1985.

\bibitem {KMO}M. A. Kenku and F. Momose, Torsion points on elliptic curves
defined over quadratic fields, Nagoya Math. J. \textbf{109} (1988) 125-149.

\bibitem {N1}B. Naskrecki, Divisibility sequences of polynomials and heights
estimates, N.Y. J. Math. \textbf{22} (2016) 989-1020.

\bibitem {N2}B. Naskrecki and M. Streng. Primitive divisors of elliptic
divisibility sequences over function fields with constant j-invariant. J. Num.
Theory \textbf{213} (2020), 152-186.

\bibitem {ReneS}R. Schoof, Elliptic curves over finite fields and the
computation of square roots mod $p$, Math. Comp. \textbf{44} (170) (1985), 483-494.

\bibitem {RS}R. Shipsey, Elliptic divisibility sequences, PhD Thesis,
Goldsmith's (University of London), 2000.

\bibitem {SP}I. E. Shparlinski, K. E. Stange, Character sums with division
polynomials, Canad. Math. Bull. \textbf{55 }(4) (2012), 850-857.

\bibitem {JS7}J. H. Silverman, Advanced Topics in the Arithmetic of Elliptic
Curves, Graduate Texts in Mathematics 151, Springer-Verlag, New York, 1994.

\bibitem {JS1}J. H. Silverman, The Arithmetic of Elliptic Curves 2nd Edition,
Graduate Texts in Mathematics 106, Springer Springer-Verlag, New York 2009.

\bibitem {JS5}J. H. Silverman, Computing heights on elliptic curves, Math.
Comp. \textbf{51} (183) (1988), 339-358.

\bibitem {JS3}J. H. Silverman, $p$-adic properties of division polynomials and
elliptic divisibility sequences, Math. Ann., \textbf{332 }(2) (2005), 443-471,
addendum 473-474.

\bibitem {JS4}J.H. Silverman, N. Stephens, The sign of an elliptic
divisibility sequence, J. Ramanujan Math. Soc. \textbf{21} (1) (2006), 1-17.

\bibitem {JS2}J. H. Silverman and J. Tate, Rational Points on Elliptic Curves,
Undergraduate Texts in Mathematics, Springer, 1992.

\bibitem {SU}A. V. Sutherland, Constructing elliptic curves over finite fields
with prescribed torsion, Math. Comp. \textbf{81} (278) (2012), 1131-1147.

\bibitem {CS}C. S. Swart, Elliptic curves and related sequences, PhD Thesis,
Royal Holloway (University of London), 2003.

\bibitem {MW1}M. Ward, Memoir on elliptic divisibility sequences, Amer. J.
Math., \textbf{70} (1948), 31-74.

\bibitem {MW2}M. Ward, The law of repetition of primes in an elliptic
divisibility sequences, Duke Math. J. \textbf{15} (1948), 941-946.\bigskip
\end{thebibliography}
\end{document}